\documentclass[12pt,reqno]{amsart}
\usepackage{latexsym,amssymb,amscd}
\usepackage{a4wide}
\usepackage{lscape}
\usepackage{subfigure}
\usepackage{color,diagram}

\def\NZQ{\Bbb}               
\def\NN{{\NZQ N}}

\def\Bu{{\mathfrak{Bu}}}
\def\Sc{{\mathfrak{Sc}}}
\def\BuG{{\mathfrak{BuG}}}

\def\kk{{\Bbbk}}
\def\Fc{{\mathcal{F}}}

\def\B'c{{\mathcal{B'}}}
\def\U'c{{\mathcal{U'}}}

\def\opn#1#2{\def#1{\operatorname{#2}}} 
\opn\chara{char}
\opn\length{\ell}
\opn\projdim{proj\,dim}
\opn\injdim{inj\,dim}
\opn\ini{in}
\opn\rank{rank}
\opn\Tiefe{Tiefe}
\opn\grade{grade}
\opn\height{height}
\opn\embdim{emb\,dim}
\opn\codim{codim}

\opn\Tr{Tr}
\opn\bigrank{big\,rank}
\opn\superheight{superheight}\opn\lcm{lcm}
\opn\trdeg{tr\,deg}%
\opn\reg{reg}
\opn\lreg{lreg}
\opn\deg{deg}
\opn\lcm{lcm}
\opn\syz{syz}
\opn\Set{Set}
\opn\Max{Max}

\opn\div{div}
\opn\Div{Div}
\opn\cl{cl}
\opn\Cl{Cl}
\opn\Spec{Spec}
\opn\Supp{Supp}
\opn\supp{supp}
\opn\Sing{Sing}
\opn\Ass{Ass}
\opn\Min{Min}
\opn\bmo{\mathbf{m}}
\opn\XT{X_{Taylor}}
\opn\Ann{Ann}
\opn\Rad{Rad}
\opn\Soc{Soc}
\opn\MinGen{\mathcal{G}}
\opn\Ker{Ker}
\opn\Coker{Coker}
\opn\Im{Im}
\opn\Hom{Hom}
\opn\Tor{Tor}
\opn\Ext{Ext}
\opn\End{End}
\opn\Aut{Aut}
\opn\id{id}

\opn\nat{nat}
\opn\GL{GL}
\opn\SL{SL}
\opn\mod{mod}
\opn\ord{ord}
\opn\depth{depth}
\opn\set{set}
\opn\Shad{Shad}
\opn\mdeg{mdeg}
\opn\aff{aff}
\opn\con{conv}
\opn\relint{relint}
\opn\st{st}
\opn\lk{lk}
\opn\cn{cn}
\opn\core{core}
\opn\vol{vol}
\opn\Cut{Cut}
\opn\Mon{Mon}
\opn\Buch{Buch}
\opn\LCM{LCM}
\opn\hull{hull}
\opn\gr{gr}

\def\pot#1#2{#1[\kern-0.28ex[#2]\kern-0.28ex]}

\opn\dirlim{\underrightarrow{\lim}}
\opn\invlim{\underleftarrow{\lim}}
\def\pnt{{\raise0.5mm\hbox{\large\bf.}}}

\def\Implies{\ifmmode\Longrightarrow \else
     \unskip${}\Longrightarrow{}$\ignorespaces\fi}
\def\implies{\ifmmode\Rightarrow \else
     \unskip${}\Rightarrow{}$\ignorespaces\fi}
\def\iff{\ifmmode\Longleftrightarrow \else
     \unskip${}\Longleftrightarrow{}$\ignorespaces\fi}

\let\:=\colon
  \makeatletter
  \CheckCommand*\refstepcounter[1]{\stepcounter{#1}%
      \protected@edef\@currentlabel
       {\csname p@#1\endcsname\csname the#1\endcsname}%
  }
  \renewcommand*\refstepcounter[1]{\stepcounter{#1}%
    \protected@edef\@currentlabel
      {\csname p@#1\expandafter\endcsname\csname the#1\endcsname}%
  }
  \def\labelformat#1{\expandafter\def\csname p@#1\endcsname##1}
  \DeclareRobustCommand\Ref[1]{\protected@edef\@tempa{\ref{#1}}%
     \expandafter\MakeUppercase\@tempa
  }
  \makeatother

  \makeatletter
  \newcommand{\numberlike}[2]{%
     \expandafter\def\csname c@#1\endcsname{%
         \expandafter\csname c@#2\endcsname}%
  }
  \makeatother

  \def\DefaultNumberTheoremWithin{section}

  \theoremstyle{plain}
  \newtheorem{Lemma}{Lemma}
     \numberwithin{Lemma}{\DefaultNumberTheoremWithin}
     \labelformat{Lemma}{Lemma~#1}
  
     \numberwithin{Claim}{\DefaultNumberTheoremWithin}
     \numberlike{Claim}{Lemma}
     \labelformat{Claim}{Claim~#1}
  \newtheorem{Theorem}{Theorem}
     \numberwithin{Theorem}{\DefaultNumberTheoremWithin}
     \numberlike{Theorem}{Lemma}
     \labelformat{Theorem}{Theorem~#1}
  \newtheorem{Corollary}{Corollary}
     \numberwithin{Corollary}{\DefaultNumberTheoremWithin}
     \numberlike{Corollary}{Lemma}
     \labelformat{Corollary}{Corollary~#1}
  \newtheorem{Proposition}{Proposition}
     \numberwithin{Proposition}{\DefaultNumberTheoremWithin}
     \numberlike{Proposition}{Lemma}
     \labelformat{Proposition}{Proposition~#1}
  \newtheorem{Conjecture}{Conjecture}
     \numberwithin{Conjecture}{\DefaultNumberTheoremWithin}
     \numberlike{Conjecture}{Lemma}
     \labelformat{Conjecture}{Conjecture~#1}

  \theoremstyle{definition}
  \newtheorem{Definition}{Definition}
     \numberwithin{Definition}{\DefaultNumberTheoremWithin}
     \numberlike{Definition}{Lemma}
     \labelformat{Definition}{Definition~#1}

  \theoremstyle{definition}
  
     \numberwithin{Question}{\DefaultNumberTheoremWithin}
     \numberlike{Question}{Lemma}
     \labelformat{Question}{Question~#1}

  \theoremstyle{definition}
  
     \numberwithin{Problem}{\DefaultNumberTheoremWithin}
     \numberlike{Problem}{Lemma}
     \labelformat{Problem}{Problem~#1}

  \theoremstyle{remark}
  \newtheorem{Remark}{Remark}
     \numberwithin{Remark}{\DefaultNumberTheoremWithin}
     \numberlike{Remark}{Lemma}
     \labelformat{Remark}{Remark~#1}
  \theoremstyle{remark}

  \newtheorem{Example}{Example}
     \numberwithin{Example}{\DefaultNumberTheoremWithin}
     \numberlike{Example}{Lemma}
     \labelformat{Example}{Example~#1}

\let\epsilon=\varepsilon
\let\phi=\varphi
\let\kappa=\varkappa
\title{The Buchberger resolution}
  \author{Anda Olteanu \and Volkmar Welker}
  
 \thanks{The first author was supported by a grant of the Romanian Ministry of Education, CNCS-UEFISCDI, project number PN-II-RU-PD-2012-3-0235 and the second author was partially supported by MSRI}
\address{Fachbereich Mathematik und Informatik, Philipps-Universit\"at Marburg, 35032 Marburg, Germany}\email{welker@mathematik.uni-marburg.de}

\address{University Politehnica of Bucharest, Faculty of Applied Sciences,
Splaiul Independen\c tei, No.
313, 060042, Bucharest, Romania}
\address{Simion Stoilow Institute of Mathematics of the Romanian Academy, Research group of the project PD-3-0235, P.O.Box 1-764, Bucharest 014700, Romania}\email{olteanuandageorgiana@gmail.com} 
\begin{document}

\maketitle
\begin{abstract} We define the Buchberger resolution, which is a graded free
  resolution of a monomial ideal in a polynomial ring. Its construction
  uses a generalization of the Buchberger graph and encodes much of the 
  combinatorics of the Buchberger algorithm. The Buchberger resolution 
  is a cellular resolution that when it is minimal coincides with the Scarf resolution.
  The simplicial complex underlying the Buchberger resolution is of interest 
  for its own sake and its combinatorics is not fully understood.
  We close with a conjecture on the clique complex of the Buchberger graph.\\
  
  \textbf{Keywords}: Monomial resolution, cellular resolution, Scarf complex, minimal free resolution.\\ 

\textbf{MSC}: Primary 13D02 Secondary 05C10,  13C14.
\end{abstract}

\section*{Introduction}

Let $S=\kk[x_1,\ldots,x_n]$ be the polynomial ring in $n$ variables over a 
field $\kk$ and $I$ a monomial ideal of $S$. The construction of the minimal 
graded free resolution of $I$ over $S$ is an important problem in 
commutative algebra. Even though there are good algorithms that 
construct the minimal graded free resolution for general monomial
ideals there is no combinatorial construction of the minimal free
graded resolution known. In this paper we provide a construction of a 
new free resolution for all monomial ideals and then identify the 
ideals for which it is minimal. We call this resolution the 
Buchberger resolution since its combinatorics is
derived from the Buchberger algorithm, an idea first employed in \cite{MS1}
in the three variable case. The Buchberger resolutions is a cellular 
resolution. Roughly speaking a cellular resolution is given by 
simplicial complex (or more generally CW-complex) with a labelling of 
its vertices by
generators of a monomial ideal and an induced labelling of its simplices by the lcms of their vertices.
If the conditions from \cite[Lem. 2.2]{BPS} are fulfilled, then the homogenized
simplicial chain complex of the simplicial complex with coefficients in the polynomial ring
supports a free resolution of the ideal.

One prominent example of a cellular is the Taylor resolution \cite{T} which is defined 
for arbitrary monomial ideals and supported by the full simplex on the set
$\MinGen(I)$ of minimal monomial generators of $I$ -- rarely the Taylor resolution is minimal.
In Theorem 3.2 from \cite{BPS} it is shown that there is a subcomplex of the simplex on 
$\MinGen(I)$ -- the Scarf complex -- which for example supports a minimal 
free resolution for generic monomial ideals, but that it does not even support a free resolution 
in general. In \cite{MS1} the
Buchberger graph of a monomial ideal is studied and, in the case of three variables,
its planar embeddings are used to define a simplicial complex supporting a minimal free 
resolution for strongly generic ideals. Our Buchberger resolution will be supported on a 
simplicial complex that generalizes the Scarf complex and planar map of a Buchberger graph. It is shown to
coincide with the Scarf complex and planar map of the Buchberger graph in cases where the Buchberger resolution is minimal.
For references to other constructions of cellular resolutions we refer the reader to \cite{P}.

The paper is organized as follows: In Section \ref{sec1} we fix the notation 
and recall the concept of the Buchberger graph associated to a monomial ideal
from \cite{MS1}.
Then we define the Buchberger complex and show that it is a 
contractible simplicial complex. Therefore, by inductive reasoning and
\cite[Lem. 2.2]{BPS}, it supports a graded free resolution, which we
call the Buchberger resolution. 
Section \ref{sec2} is devoted to the relation between the Buchberger resolution and 
the Scarf complex. The Buchberger resolution turns out to be minimal 
precisely when the Scarf complex is a resolution.

\section{Cellular resolutions, the Buchberger graph and the Buchberger complex}
\label{sec1}

We will start this section by recalling the notion of a cellular resolution of a monomial ideal $I$
in the polynomial ring $S=\kk[x_1,\ldots,x_n]$ over a field $\kk$
in its original definition from \cite{BPS}.

Let $\Delta$ be a simplicial complex whose vertices are labeled bijectively 
by the monomials from the set of minimal monomial generators $\MinGen(I)$ of $I$.
Each face $F\in\Delta$ is labeled by the least common multiple of its 
vertices, which we denote by $\bmo_F$.
The multidegrees of these monomials define an $\NN^n$-grading of $\Delta$.
Let $\Fc_{\Delta}$ be the $\NN^n$-graded chain complex of 
$\Delta$ over $S$ with differentials homogenized with respect to
the grading (see \cite{BPS} for details).

\begin{Proposition}[{\cite[Lemma 2.2]{BPS}}]
   \label{BPScriterion}
   The complex $\Fc_\Delta$ is exact and defines a free resolution of $I$
   if and only if, for every monomial 
   $m$, the simplicial
   complex $\Delta[m] = \{F \in\Delta\ |\bmo_F \mbox{ divides }m\}$ is 
   empty or acyclic over $\kk$.
\end{Proposition}

If the complex $\Fc_{\Delta}$ is exact then it is called
\textit{the resolution supported by the (labeled) simplicial complex} $\Delta$.
Moreover, if $\Fc_\Delta$ is exact, one may determine whether it is a 
minimal free resolution of $I$ or not.

\begin{Proposition}[{\cite[Remark 1.4]{BS}}] \label{BPSminimal}
  Let $\Fc_{\Delta}$ be a free resolution of the monomial ideal $I$ supported by the
  labeled simplicial complex $\Delta$. Then $\Fc_\Delta$ is a minimal free resolution
  if and only if any two 
  comparable faces $F\subset G$ of the complex $\Delta$ have distinct 
  degrees, that is $\bmo_F \neq \bmo_G$.
\end{Proposition}

Next we describe the construction of the Buchberger graph from
\cite{MS} and its connection to minimal free resolutions.

We define a partial order $<$ on $\NN^n$ as follows:
for $\mathbf{a},\mathbf{b}\in\NN^n$, where $\mathbf{a}=(a_1,\ldots,a_n)$,
$\mathbf{b}=(b_1,\ldots,b_n)$ one says that 
$\mathbf{a}<\mathbf{b}$ if $a_i<b_i$ for all $1\leq i\leq n$ such
that $b_i \neq 0$ and $a_i=b_i$ if $b_i=0$. Given a monomial 
$u=x_1^{u_1}\cdots x_n^{u_n}$ in $S$, we denote by $\mdeg(u)$ its
multidegree, that is $\mdeg(u)=(u_1,\ldots,u_n)$. For two monomials
$u,v\in S$, one says that $u$ \textit{properly divides} $v$ if $u$ divides
$v$ and $\mdeg(u)<\mdeg(v)$. We will denote this by $u\mid_p v$.

In \cite{MS1}  Miller and Sturmfels associate to any monomial ideal its
Buchberger graph which first appeared under this name in \cite{MS}. 

\begin{Definition} The \textit{Buchberger graph} of the monomial ideal $I$,
  denoted by $\BuG(I)$, is the graph on vertex set $\MinGen(I)$ with edges
  $F = \{m, m'\}$ for distinct monomials $m,m' \in \MinGen(I)$ such that
  there is no monomial $m'' \in \MinGen(I)$ which properly divides the least
  common multiple $\bmo_F$.
\end{Definition}

The Buchberger graph plays an important role in Gr\"obner basis theory
\cite{MS}, and also appears in the study of special classes of monomial
ideals such as strongly generic ideals.

\begin{Definition} A monomial ideal $I\subseteq S$ is called
  \textit{strongly generic} if, for any two monomials
  $u,v\in \mathcal{G}(I)$, $u=x_1^{u_1}\cdots x_n^{u_n}$ and
  $v=x_1^{v_1}\cdots x_n^{v_n}$, the condition
  ($u_i\neq v_i$ or $u_i=v_i=0$ for all $1\leq i\leq n$)
  is fulfilled.
\end{Definition}

In the case of polynomial rings in three variables, there is a connection
between the minimal graded free resolution of strongly generic ideals
and their corresponding Buchberger graphs.

\begin{Proposition}[{\cite[Proposition 3.9]{MS}}]
 Let $I\subseteq \kk[x,y,z]$ be a strongly generic ideal. Then
 $\BuG(I)$ is planar and connected.
\end{Proposition}

A planar graph $G$ together with a (sufficiently nice) embedding of $G$ into the
plane $\mathbb{R}^2$ is called a \textit{planar map}. The vertices of
the embedded graph can be considered as  $0$-cells, the edges as $1$-cells 
and the regions bounded by $G$ as $2$-cells. Viewed from this perspective the 
embedding defines a CW-complex. In \cite[pp. 51]{MS} it is shown that this 
complex supports a free resolution of some monomial ideal. Moreover:

\begin{Theorem}[{\cite[Theorem 3.11]{MS}}]
\label{BuRes}
 Let $I\subseteq \kk[x,y,z]$ be a strongly generic ideal. Then
 any planar map of $\BuG(I)$ with vertices labelled by the generators
 and edges and faces labelled by the lcms of its vertices supports a 
 minimal free resolution of $I$.
\end{Theorem}

It is natural to ask:
  {\it Can this result be extended to strongly generic ideals in a
  polynomial ring in $n$ variables, $S=\kk[x_1,\ldots,x_n]$?}

In order to answer this question, we generalize the Buchberger graph to a
suitable simplicial complex and then prove that this simplicial complex
supports a free resolution.

\begin{Definition}
  Let $I\subseteq\kk[x_1,\ldots,x_n]$ be a monomial ideal. 
  The \textit{Buchberger complex} $\Bu(I)$ of $I$, is the collection of 
  all subsets $F$ of $\MinGen(I)$ such that no $u \in \MinGen(I)$ properly divides $\bmo_F$.
\end{Definition}

The set system $\Bu(I)$ is indeed a simplicial complex. To see this assume
$F\in \Bu(I)$ and $G\subsetneq F$, $G\notin\Bu(I)$. Then $\bmo_G\mid \bmo_F$. But $G\notin \Bu(I)$
and hence there is a minimal monomial generator
$u\in I$ such that $u \mid_p \bmo_G$, which implies $u\mid_p \bmo_F$, a 
contradiction.

\begin{Remark}
   \begin{itemize}
      \item The $1$-skeleton $\Bu(I)^{\langle 1 \rangle}$ of the Buchberger complex 
         $\Bu(I)$ of the monomial ideal $I$ is the Buchberger graph $\BuG(I)$ 
         of $I$. 
      \item If $I$ is a squarefree
         monomial ideal, then its Buchberger complex is the full simplex 
         $2^{\MinGen(I)}$.
   \end{itemize}
\end{Remark}

\begin{Example}
  \label{ex}
  Let $I=(x_1^2,x_2^2,x_3^2,x_1x_3,x_2x_4)\subseteq \kk[x_1,x_2,x_3,x_4]$. 
  It is easy to see that
  $x_1x_3\mid_p x_1^2x_3^2=\lcm(x_1^2,x_3^2)$, 
  therefore $\{x_1^2,x_3^2\}\notin \Bu(I)$. In this
  case the Buchberger complex has two facets (see Figure 1), namely

  $$\Bu(I)=\langle\{x_1^2,x_2^2,x_1x_3,x_2x_4\},
                  \{x_3^2,x_2^2,x_1x_3,x_2x_4\}\rangle.$$

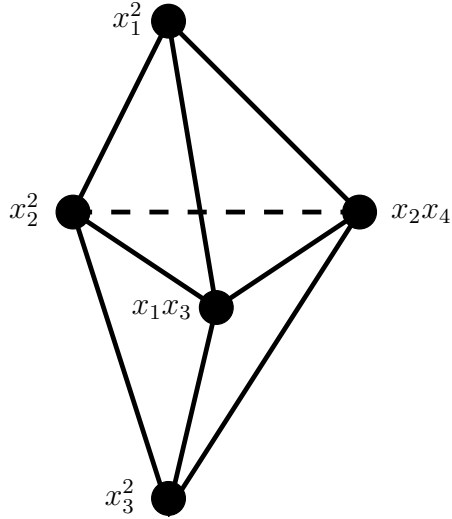
\begin{figure}\label{f1}
  \begin{picture}(0,0)%
     \includegraphics{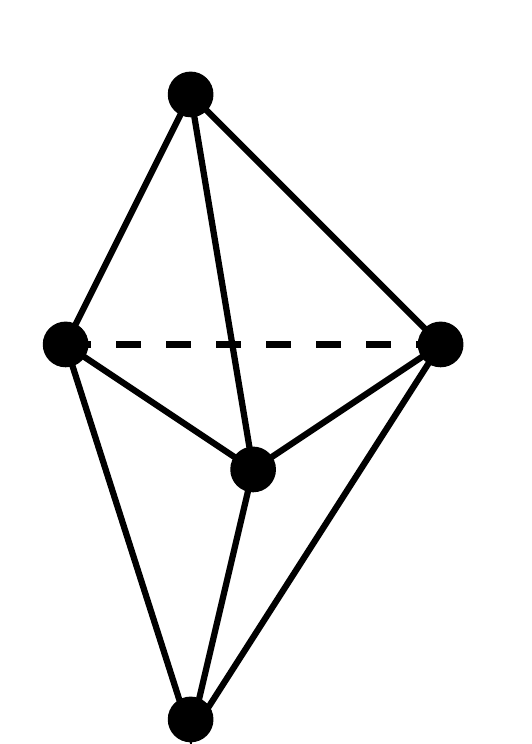}%
  \end{picture}%
  \setlength{\unitlength}{3947sp}%
  \begingroup\makeatletter\ifx\SetFigFont\undefined%
  \gdef\SetFigFont#1#2#3#4#5{%
   \reset@font\fontsize{#1}{#2pt}%
   \fontfamily{#3}\fontseries{#4}\fontshape{#5}%
   \selectfont}%
  \fi\endgroup%
  \begin{picture}(2537,3563)(2086,-3474)
      \put(2000,-1600){\makebox(0,0)[lb]{\smash{{\SetFigFont{12}{14.4}{\rmdefault}{\mddefault}{\updefault}{\color[rgb]{0,0,0}$x_2^2$}%
}}}}
      \put(2650,-400){\makebox(0,0)[lb]{\smash{{\SetFigFont{12}{14.4}{\rmdefault}{\mddefault}{\updefault}{\color[rgb]{0,0,0}$x_1^2$}%
}}}}
      \put(2600,-3400){\makebox(0,0)[lb]{\smash{{\SetFigFont{12}{14.4}{\rmdefault}{\mddefault}{\updefault}{\color[rgb]{0,0,0}$x_3^2$}%
}}}}
      \put(2770,-2200){\makebox(0,0)[lb]{\smash{{\SetFigFont{12}{14.4}{\rmdefault}{\mddefault}{\updefault}{\color[rgb]{0,0,0}$x_1x_3$}%
}}}}
      \put(4400,-1600){\makebox(0,0)[lb]{\smash{{\SetFigFont{12}{14.4}{\rmdefault}{\mddefault}{\updefault}{\color[rgb]{0,0,0}$x_2x_4$}%
}}}}
    \end{picture}%
    \caption{Buchberger Complex of $I = (x_1^2,x_2^2,x_3^2,x_1x_3,x_2x_4)$}
  \end{figure}
\end{Example}

In order to verify the conditions from \ref{BPScriterion} for the Buchberger complex, we
have to introduce some combinatorial constructions and exhibit some of their
properties.

Let $P$ be a finite poset and $\Delta(P)$ its order complex, that is the 
simplicial complex whose simplices are the chains in the poset $P$. Through
the geometric realization of $\Delta(P)$ we then can speak of homotopy type,
homotopy equivalence and contractability of posets.
To a monomial ideal $I$ we associate its \textit{lcm-lattice} $L_I$, which is 
the poset on the least common multiples $\bmo_F$ of subsets 
$F \subseteq \MinGen(I)$ ordered by the divisibility \cite{GPW}. It is easily 
checked that $L_I$ indeed is a lattice with least element 
$1 = \lcm(\emptyset)$.
For a monomial $m \in L_I$ we write $(1,m)$ for the open interval of all 
$m' \in L_I$ such that $1 < m' < m$.

In the next two lemmas we collect basic topological properties of the lcm lattice that will
prove crucial for our first main result.

\begin{Lemma}\label{LCM}
  Let $m\in L_I$. If there is a monomial
  $w \in \MinGen(I)$ such that $w \mid_p m$ then the interval $(1,m)$ is 
  contractible.
\end{Lemma}
\begin{proof}
  Let $w\in \MinGen(I)$ be a monomial which properly divides $m$. We consider 
  the map $f: (1,m) \rightarrow (1,m)$, defined by $f(u)=\lcm(u,w)$. Since 
  $w$ properly divides $m$ it follows
  that $f(u) \in (1,m)$ for all $u \in (1,m)$. By construction 
  for $u,u' \in (1,m)$ and $u \leq u'$ it follows that $f(u) \leq f(u')$. 
  Again by definition $f(f(u)) = f(u)$ and hence
  $f$ is a closure operator on the poset $(1,m)$. By \cite[Cor. 10.11]{Hand} this 
  shows that the interval $(1,m)$ 
  and its image under $f$ are homotopy equivalent. Moreover, the image of $(1,m)$ under
  $f$ has $w$ as its unique smallest element.
  Therefore, the order complex of the image of $f$ is a cone over $w$ and 
  hence contractible. Thus $(1,m)$ is contractible itself.
\end{proof}

\begin{Lemma}\label{LCM-1}
  Let $P_I$ be the poset of all monomials $m \in L_I \setminus\{1\}$ such that 
  there is no monomial $w\in \MinGen(I)$ which properly divides $m$. 
  Then $P_I$ is contractible.
\end{Lemma}
\begin{proof}
  By \ref{LCM} we know that for each $m \in L_I$ for which there is an 
  $w \in \MinGen(I)$ that properly divides $m$ the interval $(1,m)$ is contractible.
  Let $M$ be the set of all $m \in L_I$ with this property. Then 
  a simple application of the Quillen Fiber Lemma (see \cite[Thm. 10.11]{Hand} show
  that $L_I \setminus ( M \cup \{1\})$ is homotopy equivalent to $L_I \setminus \{1\}$.
  But $L_I \setminus \{1\}$ has $\lcm \MinGen(I)$ as its unique maximal element and
  hence its geometric realization is a cone. Therefore, it is contractible. But this
  shows that $P_I$ must be contractible.
\end{proof}

For the proof of our first main theorem we need yet another concept from topological 
combinatorics.
We call a subset $B$ of a poset $P$ bounded if there are elements $m$ and $m'$
in $P$ such $m \leq n \leq m'$ for all $n \in B$.
A crosscut $A$ in $P$ is an antichain such that for every chain 
$C\subseteq P$, there exists an element $a\in A$ such that $a$ is comparable 
with all elements in $C$ and every bounded subset $B$ of $A$ has an infimum 
and a supremum. Let $\Gamma(P,A)$ be the collection of subsets of $A$ 
that are bounded. It is easily seen that $\Gamma(P,A)$ is a simplicial complex.
It is called the crosscut complex of $P$ and $A$. The relation between the crosscut 
complex $\Gamma(P,A)$ and the topology of $P$ is given by the homotopy version of
Rota's Crosscut Theorem, which says (see \cite[Thm. 10.8]{Hand}) that for a finite poset 
$P$ and a crosscut $A$ in $P$
the complexes $\Gamma(P,A)$ and $\Delta(P)$ are homotopy equivalent.

\begin{Theorem}
  Let $I$ be a monomial ideal. Then the Buchberger complex $\Bu(I)$ supports a cellular 
  resolution of $I$.
\end{Theorem}
\begin{proof}
  We will first verify that $\Bu(I)$ is contractible and then show that this
  implies the criterion from \ref{BPScriterion} for a simplicial complex to
  support a free resolution.

  Let $P_I$ be the poset of all monomials $m\in L_I \setminus \{1\}$ such that there is no
  monomial $w\in \MinGen(I)$ which properly divides $m$.
  If $m \in P_I$ and $m' \in L_I$ divides $m$ then
  $m' \in P_I$. Otherwise there is 
  $m'' \in L_I$ which properly divides $m'$ and hence properly divides
  $m$, contradicting $m \in P_I$. Thus $P_I$ is a lower order ideal in $L_I \setminus \{1\}$.
  As $L_I$ is a lattice this implies that any subset of $P_I$ which is bounded from below 
  (resp. above) has an infimum (resp. supremum) in $P_I$.

  Let $A = \MinGen(I)$. Then $A \subseteq P_I$ and since $A$ is a generating
  set of the monomial ideal $I$, for every chain in $P_I$ there is an element
  from $A$ comparable with all elements from the
  chain. Since we know that any bounded subset of $P_I$ has an infimum and supremum in
  $P_I$ is follows that $A$ is a crosscut in $P_I$. Moreover, it follows that the
  crosscut complex $\Gamma(P_I,A)$ is the Buchberger complex $\Bu(I)$.
  Therefore, by \ref{LCM-1} and the Crosscut Theorem \cite[Thm. 10.8]{Hand} it follows that 
  $\Bu(I)$ is contractible.

   Set $\Delta = \Bu(I)$ and let $m$ be some monomial such that $\Delta[m]$
   is non-empty. Then $\Delta[m]$ consists of all
   subsets $F \in \Bu(I)$ of $\MinGen(I)$ such that $\bmo_F$ divides $m$.
   Thus $F \in \Delta[m]$ if and only if $\bmo_F$ divides $m$ and there is no
   $u \in \MinGen(I)$ that properly divides
   $\bmo_F$. Let $J_m$ be the monomial ideal generated by all
   $u \in \MinGen(I)$ such that $u$ divides $m$.
   Then $F \in \Delta[m]$ if and only if $F \subseteq \MinGen(J_m)$ and 
   there is no $u \in \MinGen(J_m)$
   that properly divides $\bmo_F$. This implies that $F \in \Delta[m]$ 
   if and only if $F \in
   \Bu(J_m)$ and hence $\Delta[m] = \Bu(J_m)$. But then by the first 
   part of the proof it follows that
   $\Delta[m]$ is contractible and hence acyclic.

   Thus we have verified the conditions from \ref{BPScriterion} for 
   $\Delta = \Bu(I)$ and hence
   $\Bu(I)$ supports a cellular resolution of $I$.
\end{proof}

We call the complex $\mathcal{F}_{\Bu(I)}$ supported on the Buchberger 
complex $\Bu(I)$ the \textit{Buchberger resolution} of the monomial ideal $I$.
Next, we characterize when the Buchberger resolution is minimal.

\begin{Proposition}\label{minimal}
    Let $I$ be a monomial ideal and $\Bu(I)$ its Buchberger complex. Then the Buchberger resolution is a
    minimal free resolution of $I$ if and only if whenever $F,G\subset \MinGen(I)$ are such that
    $\bmo_F=\bmo_G$, then there is a monomial $u\in \MinGen(I)$ such that $u$
    properly divides $\bmo_F=\bmo_G$.
\end{Proposition}
\begin{proof}
    It is clear that if whenever $F,G\subset \MinGen(I)$ are such that $\bmo_F=\bmo_G$, there is a
    minimal monomial generator $u\in \MinGen(I)$ such that $u$ properly divides $\bmo_F=\bmo_G$
    it follows that $F,G \not\in \Bu(I)$.
    Then by \ref{BPSminimal} the Buchberger resolution is a minimal free resolution of $I$.

    Assume that the Buchberger resolution is a minimal free resolution of $I$. Let
    $F,G \in \Bu(I)$ such that $\bmo_F = \bmo_G$.
    If $F\subseteq G$ or $G \subseteq F$ then \ref{BPSminimal} implies that $F = G$ and we are done.
    Thus we can assume that $F$ and $G$ are incomparable with respect to inclusion.
    Thus we can choose $m_\alpha\in F\setminus G$ and $m_\beta \in G \setminus F$. 
    Then $\{m_\alpha,m_{\beta}\}\in \Bu(I)$ for all
    $m_{\beta}\in G$ since otherwise there exists $u\in \mathcal{G}(I)$ with $u\mid_p
    \lcm(m_{\alpha},m_{\beta})\mid \bmo_F$ and $F\notin \Bu(I)$, a contradiction. Therefore,
    $\{m_{\alpha}\}\cup G\in \Bu(I)$. Since the resolution is minimal and $\bmo_G=\bmo_{\{m_{\alpha}\}\cup G}$, one has $G$ and $\{m_{\alpha}\}\cup G$ are not
    faces in $\Bu(I)$, a contradiction.
\end{proof}

Now we return to the monomial ideal from \ref{ex} and explicitly construct the
Buchberger resolution. Moreover, one can easy check or use \ref{minimal} in order to prove
that the resolution is minimal.

\begin{Example}
  Let $I=(x_1^2,x_2^2,x_3^2,x_1x_3,x_2x_4)\subseteq \kk[x_1,x_2,x_3,x_4]$. Then counting faces in Figure 1 provides 
  \[\begin{array}{ccccccccccccc}
     0&\rightarrow&S^2&\rightarrow&S^7&\rightarrow&S^9&\rightarrow&S^5&\rightarrow&I&\rightarrow&0
  \end{array}\]
  as its minimal free resolution over $S$. Moreover, the differentials
  are easily written out by homogenizing the simplicial differential.
\end{Example}

For a monomial $m \in L_I$ we say that $m$ is a Buchberger degree if there is 
no $u \in \MinGen(I)$ that properly divides $m$. If $m$ is a Buchberger 
degree for $I$ then we denote by $B_m(I)$ the poset of all subsets $A$ of 
$[n] := \{1,\ldots, n\}$ such that for some $m' \in (1,m)$ the set $A$ is the 
set of indices $i \in [n]$ such that the exponents of $x_i$ in $m'$ and $m$ 
coincide.

\begin{Corollary}
  \label{co:bool}
  Let $I$ be a monomial ideal and $m \in L_I$. Then either 
  $\beta_{i,m} = 0$ if $m$ is not a Buchberger
  degree or $\beta_{i,m} = \dim_\kk \widetilde{H}_{i-1}(B_m(I);\kk)$.
\end{Corollary}
\begin{proof}
  By \ref{minimal} we know that $\beta_{i,m} = 0$ if $m$ is not a Buchberger 
  degree. Assume $m$ is a Buchberger degree. 
  Consider the map $f: (1,m) \rightarrow B_m(I)$
  that sends a monomial $m'$ to the set of indices $i$ for which the exponent 
  of $x_i$ in $m$ and $m'$ coincide. For any $A \in B_m(I)$ the lower fiber 
  $f^{-1} (B_m(I)_{\leq A})$
  consists of all $m' \in (1,m)$ for which the set of indices such that the 
  exponents of $x_i$ in
  $m$ and $m'$ coincide is a subset of $A$. Since $A \neq [n]$ it follows 
  that the lcm of all
  elements of $f^{-1}(B_m(I)_{\leq A})$ is an element of $(1,m)$. Hence the 
  fiber
  $f^{-1}(B_m(I)_{\leq A})$ has a unique maximal element. Thus its order 
  complex is a cone and hence
  the fiber is contractible. Then the Quillen fiber lemma says that $(1,m)$ 
  and $B_m(I)$ are homotopy equivalent. In particular, they have the same 
  reduced homology. Since $\beta_{i,m} =
  \dim_{\kk}\widetilde{H}_{i-1}((1,m);\kk)$ the result follows.
\end{proof}

Note that if $m$ is a monomial on $n$ variables $x_1,\ldots, x_n$ then 
$B_m(I)$ is a subset of the Boolean lattice $2^{[n]}$ on $[n]$. 
Indeed $B_m(I) \cup \{ \emptyset,[n] \}$ is a lattice whose join
operation coincides with the union of sets. It is atomic and its atoms are 
the images of the generators of $I$ dividing $m$. Now \ref{co:bool} shows 
that for a monomial $m$ on $n$ variables the Betti number $\beta_{i,m}$ is 
bounded from above by the maximal rank of the $(i-1)$\textsuperscript{st}
homology of an atomic join sublattice of the Boolean lattice on $n$ elements. 
It is also worthwhile to study what can be said about the poset of all $m$
that are Buchberger degrees. Does this poset have an interesting structure ? 
By the above corollary this question relates to the even more challenging
poset of all $m$ for which $\beta_{i,m}$ is nonzero for some $i$. A study of 
this poset was initiated in \cite[Sec. 8]{TV}, where first results 
can be found.

\section{Minimality and relation to the Scarf complex}
\label{sec2}

In this section we are interested in determining the connections between 
the Buchberger complex and the Scarf complex and the related question of when
the Buchberger complex defines a minimal free resolution. 
As before, let $I$ be a monomial 
ideal with minimal monomial generating set $\mathcal{G}(I)=\{m_1,\ldots,m_r\}$.
We recall the definition of the Scarf complex (see \cite{BPS}).

\begin{Definition}
  The \textit{Scarf complex} $\Sc_I$ of $I$ is the collection of all subsets 
  of $\MinGen(I)$ whose least common multiple is unique:
  \[
     \Sc_I=\{\sigma\subseteq\{1,\ldots,r\}: \bmo_\sigma = 
            \bmo_\tau \Rightarrow (\sigma =\tau)\}.
  \]
  We call the complex $\Fc_{\Sc(I)}$ of free $S$-modules supported on the Scarf complex $\Sc(I)$ 
  the \textit{algebraic Scarf complex} of the monomial ideal $I$.
\end{Definition}

Note that in general $\Fc_{\Sc(I)}$ is 
not a free resolution of $I$ (see below for
details).

The following remark clarifies the connection between the Scarf complex, the 
Buchberger graph and the Buchberger complex.

\begin{Remark} \label{Scarf}
  \begin{enumerate}
    \item It is easily seen that 
       $\Sc_I^{\langle 1\rangle}\subseteq\BuG(I)$, but the converse does not
       hold in general.
    \item One has that $\Sc_I\subseteq \Bu(I)$. Indeed, let us assume by 
       contradiction that $F\in\Sc_I$ and $F\notin \Bu(I)$. Then there must 
       be a monomial
       $u\in \MinGen(I)$ such that $u \mid_p \bmo_F$. But, in this case, 
       $\lcm(u,\bmo_F)=\bmo_F$, that is
       $F\notin \Sc_I$, a contradiction.
  \end{enumerate}
\end{Remark}

We can characterize the minimality of the Buchberger resolution in terms of 
the Scarf  complex.

\begin{Proposition}\label{Buch Scarf}
  Let $I$ be a monomial ideal in $S$. The following are equivalent:
  \begin{itemize}
	\item[(a)] $\Fc_{\Bu(I)}$ is a minimal resolution of $I$.
	\item[(b)] The Scarf complex $\Sc(I)$ and the Buchberger 
          complex $\Bu(I)$ coincide.
  \end{itemize}
\end{Proposition}

\begin{proof}
  \noindent {\sf (a)$\Rightarrow$(b):}
  Let $\Fc_{\Bu(I)}$ be a minimal free resolution. Assume that there is a 
  face $F\in \Bu(I)$ which is not in the Scarf complex $\Sc(I)$. 
  The latter implies that there is $G\subseteq
  \MinGen(I)$ such that $\bmo_F=\bmo_G$. By the minimality of the resolution, 
  \ref{minimal}, we get that there is a monomial $u\in \MinGen(I)$ which 
  properly divides $\bmo_F$. But this is a contradiction with
  $F\in \Bu(I)$.

  \noindent {\sf (b)$\Rightarrow$(a):} If $\Sc_I=\Bu(I)$, it is clear that 
  the least common multiple of the monomials from each face of $\Bu(I)$ 
  is unique. The statement follows.
\end{proof}

Since it is known from \cite{MS} that the Scarf complex of a monomial ideal $I$
defines a free resolution of $I$ if and only if it defines a minimal free resolution
of $I$ the following corollary is now straightforward.

\begin{Corollary}\label{Buch Scarf1}
  Let $I$ be a monomial ideal in $S$. If $\Fc_{\Bu(I)}$ is a minimal resolution of $I$, then $\Fc_{\Sc(I)}$ is a minimal resolution of $I$.
\end{Corollary}

The following example shows that the converse does not hold in general (see \cite[Thm. 5.3]{PV} for more details on when the Scarf complex
supports a minimal free resolution).

\begin{Example} Let $I=(xa,yb,zc,xyz)$ be a monomial ideal in the polynomial ring $k[x,y,z,a,b,c]$. Since $I$ is a squarefree monomial ideal, $\Bu(I)=\langle\{xa,yb,zc,xyz\}\rangle$ so $\Fc_{\Bu(I)}$ is the Taylor resolution, while the Scarf complex is $$\Sc(I)=\langle\{xa,yb,xyz\},\{xa,zc,xyz\},\{yb,zc,xyz\}\rangle$$ since $\lcm(xa,yb,zc,xyz)=\lcm(xa,yb,zc)$. One may check that $\Fc_{\Sc(I)}$ is a resolution, therefore is minimal. 
\end{Example}

A class of monomial ideals for which the algebraic Scarf complex is a
minimal resolution is that of generic monomial ideals (see \cite{BPS}[Thm. 3.2] and 
\cite[Thm. 6.13]{MS}).

\begin{Definition}
   A monomial ideal $I$ is \textit{generic} if whenever two distinct 
   minimal generators $m, m' \in \MinGen(I)$ have the same positive 
   (nonzero) degree in some variable, a third
   generator $m'' \in \MinGen(I)$ strictly divides their least common 
   multiple $\lcm( m, m')$.
\end{Definition}

One may note that every strongly generic ideal is generic. This implies 
that the algebraic Scarf complex is a minimal free resolution for strongly 
generic ideals.

Next we derive a slight generalization of results from \cite[Thm. 6.26]{MS}.
For its formulation we write $\mathfrak{m}$ for the ideal $(x_1,\ldots, x_n)$
in $S$ and $\mathfrak{m}^{\mathbf{u}+1}$ for the ideal $(x_1^{u_1+1},
\ldots, x_n^{u_n+1})$, where $\mathbf{u} = (u_1,\ldots,u_n)$ is an $n$-tuple of
nonnegative integers.

\begin{Proposition}\label{I-bar} 
  Fix a vector $\mathbf{u} = (u_1,\ldots, u_n)$ of nonnegative integers and
  an ideal $I\subseteq \kk[x_1,\ldots,x_n]$ 
  generated by monomials dividing $\mathbf{x}^{\mathbf{u}}$. Set 
  $\bar{I} = I + \frak{m}^{\mathbf{u+1}}M$, where $M$ is a monomial with 
  support in $\kk[x_{n+1},\ldots,x_m]$, $m>n$. If $I$ is generic, then 
  $\Fc_{\Bu(\bar{I})}$ is a minimal free resolution of $\bar{I}$. Moreover, $\Fc_{\Sc(\bar{I})}$ is a minimal free resolution of 
  $\bar{I}$.
\end{Proposition}

In order to prove this result, we need the following lemma:

\begin{Lemma}\label{Buch Scarf2}
  Let $I\subseteq S=\kk[x_1,\ldots,x_n]$ be a monomial ideal. 
  The following are equivalent:
  \begin{itemize}
    \item[(a)] $\Sc(I)=\Bu(I)$;
    \item[(b)] If $F\notin\Sc(I)$ there is a monomial $w\in I$ which properly divides $\bmo_F$.
\end{itemize}
\end{Lemma}
\begin{proof} {\sf (a) $\Rightarrow$ (b):}\, Let us assume that $\Sc(I)=\Bu(I)$ and let $F\notin\Sc(I)$, that is $F\notin \Bu(I)$. Therefore, there is a monomial $w\in\MinGen(I)$ which properly divides $\bmo_F$.

\noindent {\sf (b)$\Rightarrow$(a):}\, By using \ref{Scarf}(2), we only have to prove that $\Bu(I)\subseteq\Sc(I)$. Let us assume by contradiction that there is $F\in \Bu(I)$ such that $F\notin\Sc(I)$. By our assumption, since $F\notin\Sc_I$, there is a monomial $w\in I$ which properly divides $\bmo_F$. In particular, there is a monomial $w'\in \MinGen(I)$ which properly divides $\bmo_F$ which implies $F\notin \Bu(I)$, a contradiction.
\end{proof}

\begin{proof}[Proof of \ref{I-bar}]
  By \ref{Buch Scarf} and \ref{Buch Scarf2}, we only have to show that if $F\notin\Sc(\bar{I})$ there is a
  monomial $w\in \bar{I}$ which properly divides $\bmo_F$. The proof of this statement is identical to the
  implication \cite[Theorem 6.26]{MS} ``(a)$\Rightarrow$(g)". We recall it here for the sake of completeness.

  Let $F\notin\Sc(\bar{I})$, that is there is $G\subseteq\MinGen(\bar{I})$ such that $\bmo_F=\bmo_G$. Let us
  assume that $F$ is maximal with respect to the inclusion among the subsets of $\MinGen(\bar{I})$ with the
  same label $\bmo_F$. Therefore there is some $m'\in F$ such that $\bmo_F=\bmo_{F\setminus\{m'\}}$. If
  $\supp(\bmo_F)=\supp(\bmo_{F\setminus\{m'\}})$, then $m'\mid_p\bmo_F$ which ends the proof. Here, for a
  monomial $m=x_1^{a_1}\cdots x_n^{a_n}$, $\supp(m)=\{x_i\ :\ a_i\neq0\}$. Therefore we assume that
  $\supp(\bmo_F)\neq\supp(\bmo_{F\setminus\{m'\}})$. Since $\bmo_F=\bmo_{F\setminus\{m'\}}$, there is a
  monomial $m''\in F$, $m''\neq m'$, and a variable $x_k\in\supp(m'')$ such that $\deg_k(m')=\deg_k(m'')>0$. Here we denote by $\deg_k(m)$ the exponent of the variable $x_k$ in the monomial $m$.
  Since $m''\neq m'$, we must have that neither $m'$ nor $m''$ are in $\MinGen(\frak{m}^{\mathbf{u+1}}M)$.
  Therefore $m',m''\in\MinGen(I)$ and, since $I$ is generic, there is a monomial $w\in\MinGen(I)$ such that
  $w\mid_p\lcm(m',m'')$. Since $\lcm(m',m'')\mid\bmo_F$ the statement follows.
\end{proof}

The preceding results show that the Buchberger complex is an interesting object and encodes 
very well the combinatorics and algebra of monomial ideals. Nevertheless, there is an alternative object
which equally well generalizes the Buchberger graph from \cite{MS1} in the situation where it can be embedded
with triangular cells. For a monomial ideal $I$ and its Buchberger graph $\BuG(I)$ let $\Cl(\BuG(I))$ be
the simplicial complex whose simplices are the subsets of $\MinGen(I)$ that induce a clique in $\BuG(I)$ --
this construction is also known as the clique complex or Rips complex of $\BuG(I)$.

We conjecture:

\begin{Conjecture}
  For a monomial ideal $I$ the clique complex $\Cl(\BuG(I))$ of the Buchberger graph $\BuG(I)$ is
  contractible.
\end{Conjecture}

In particular, if the conjecture holds true then a simple reasoning shows that $\Cl(\BuG(I))$ also
supports a minimal free resolution of $I$. Clearly, $\Bu(I) \subseteq \Cl(\BuG(I))$ and therefore
the construction will not yield more cases where we can construct a cellular minimal free resolution.
Nevertheless, we consider the conjecture seems appealing from a combinatorial point of view. The
conjecture is supported by \ref{BuRes} in the three variable case and many experiments.

\section{Acknowledgement}

The authors thank Lukas Katth\"an for pointing out an error in the formulation of \ref{Buch Scarf1} in a previous
version of this paper.

\end{document}